\documentclass[a4paper]{article}
\usepackage{amsmath, amssymb, amsthm, graphicx}
\makeatletter
\newenvironment{tablehere}{\def\@captype{table}}{}
\newenvironment{figurehere}{\def\@captype{figure}}{}
\renewenvironment{itemize}{ 
	\begin{list}{\alph{enumi})}{\setlength{\labelwidth}{1em}\setlength{\itemsep}{0em}\setlength{\parsep}{0em}}\usecounter{enumi}
}{
	\end{list}
}
\renewenvironment{enumerate}{ 
	\begin{list}{\arabic{enumii}.}{\setlength{\labelwidth}{1em}\setlength{\itemsep}{0em}\setlength{\parsep}{0em}}\usecounter{enumii}
}{
	\end{list}
}
\makeatother
\newtheorem{theorem}{Theorem}
\newtheorem{lemma}[theorem]{Lemma}
\newtheorem{corollary}[theorem]{Corollary}
\theoremstyle{definition}

\begin{document}
\title{Enumerating and identifying semiperfect colorings of symmetrical patterns}
\author{Rene P. Felix and Manuel Joseph C. Loquias\\ \small Institute of Mathematics, University of the Philippines,\\ \small Diliman, C.P. Garcia St., 1101 Diliman, Quezon City, Philippines}
\date{March 3, 2008} 
\maketitle
 
\abstract{If $G$ is the symmetry group of an uncolored pattern then a coloring of the pattern is semiperfect if the associated color group $H$ is a subgroup of $G$ of index 2.  We give results on how to identify and enumerate all inequivalent semiperfect colorings of certain patterns.  This is achieved by treating a coloring as a partition $\{hJ_iY_i:i\in I,h\in H\}$ of $G$, where $H$ is a subgroup of index 2 in $G$, $J_i\leq H$ for $i\in I$, and $Y=\cup_{i\in I}{Y_i}$ is a complete set of right coset representatives of $H$ in $G$.  We also give a one-to-one correspondence between inequivalent semiperfect colorings whose associated color groups are conjugate subgroups with respect to the normalizer of $G$ in the group of isometries of $\mathbf{R}^n$.}

\section{Introduction}
Colors are used to represent and differentiate various chemical, physical, or geometric aspects of a symmetrical structure.  For instance, different colors may correspond to different types of atoms or to different orientations of a magnetic moment in a crystal.  These applications paved the way to the development of the theory of color symmetry.  The theory was firmly established with Shubnikov's work on antisymmetry in the 1950's \cite{Sc, S}.  In addition, scientists then were intrigued by the colored tilings of the Dutch graphic artist Escher and began to analyze them mathematically \cite{Mc}. A comprehensive discussion of the history and applications of the theory of color symmetry by Schwarzenberger can be seen in \cite{Sc}.  Color symmetry remains to be of interest until today because of its relation with crystallography.

In color symmetry, we not only look at the symmetrical pattern, but also the various ways of symmetrically coloring the pattern.  Given a colored symmetrical pattern, three groups are associated to it: the group $G$ of symmetries sending the uncolored pattern to itself, the subgroup $H$ of $G$ consisting of symmetries that induce a permutation of colors in the pattern (called the \emph{color group} associated to the coloring), and the subgroup $K$ of $H$ whose elements fix the colors (called the \emph{symmetry group of the colored pattern}). Since $H$ acts on the set $C$ of colors of the pattern, there exists a  homomorphism $f:H\rightarrow P(C)$ where $P(C)$ is the group of permutations of $C$.  The kernel of $f$ is $K$. Hence, the group of color permutations $f(H)$ is isomorphic to $H/K$ and this is usually referred to as the \emph{color permutation group} corresponding to the coloring.

Senechal outlined the development of color symmetry as applied to symmetrical patterns and posed open problems in \cite{S}.  In the same paper, she pointed out that classifying color groups and classifying colored patterns are not the same since different colored patterns may correspond to the same color group.  Roth also distinguished between the notions of equivalence for colored patterns and equivalence for color groups, and gave several illustrations in \cite{Ro1} and \cite{Ro2}.

A colored symmetrical pattern is said to be perfectly colored if $H=G$.  This is the most restrictive case since all symmetries of the uncolored pattern effect color permutations.  These colorings were first studied by Burckhardt and Van der Waerden in \cite{BVW}.  Perfect colorings of different types of patterns have been studied in detail, such as transitive tilings and patterns in the plane \cite{GS} and
hyperbolic tessalations \cite{DLPFL}.

On the other hand, nonperfect colorings have also appeared in some early works in color symmetry, such as colored patterns in \cite{WG}.  In \cite{Ro1}, Roth discussed nonperfect transitive colorings (referred to as ``partially symmetric colorings'') of certain patterns and their equivalence.  It was suggested by Senechal in \cite{S} that a systematic study of nonperfect colorings might become useful and interesting.  

Rigby encountered nonperfect colorings while studying precise colorings of the regular triangular tiling $\{3,n\}$ \cite{R}.  He coined the term ``semiperfect coloring'' to describe colorings wherein half of the direct (orientation preserving) and half of the opposite (orientation reversing) symmetries of the uncolored pattern permute the colors.  Rigby also used the term ``chirally perfect'' for colorings where all the direct symmetries of the uncolored pattern effect color permutations but none of the opposite symmetries do.  In both cases, the corresponding color group $H$ is of index 2 inside the group $G$ of symmetries of the uncolored regular triangular tiling $\{3,n\}$.  In this paper, we will look at how to obtain colorings where $[G:H]=2$.

\section{Preliminaries}
Let $G$ be a subgroup of an $n$-dimensional crystallographic group and $X$ be the set of objects in a given pattern to be colored.  Suppose
$G$ acts transitively on $X$ such that for all $x\in X$, the stabilizer of $x$ in $G$ is $\{e\}$.  If this is the case, then the $G$-orbit of an $x\in X$ is $Gx=\{gx:g\in G\}=X$ and we obtain a one-to-one correspondence between $G$ and $X$ given by $g\leftrightarrow gx$.  Hence, we can associate a partition $P=\{P_1,\ldots,P_r\}$ of $G$ with the partition $\{P_1x,\ldots,P_rx\}$ of $X$.  Given a set $C=\{c_1,\ldots,c_r\}$ of $r$ colors, we call the assignment of each color $c_i$ to $P_ix$ a \emph{coloring} of $X$ corresponding to the partition $P$.  Thus, a coloring of the pattern is treated as a partition $P$ of $G$ where each element of $P$ corresponds to a unique color. 

The group $G$ acts on the set of all partitions of $G$ by left multiplication. Denote by $H$ the stabilizer of a partition $P$ of $G$.  Hence,
$h\in H\Leftrightarrow hP=P$ and we say that $P$ is an \emph{$H$-invariant partition of $G$}.  Geometrically, an element $h\in H$ is said to \emph{permute} the colors in a coloring of $X$ ($h$ is also called a \emph{color symmetry}).  This means that all objects in $X$ of a given color is mapped by $h$ onto objects in $X$ of a single color.  That is, we can associate $h$ to a permutation of the set of colors.  When $[G:H]=1$ or $H=G$, a coloring associated with $P$ is called \emph{perfect}.  If $[G:H]=2$, we call a coloring associated with $P$ \emph{semiperfect}.

To obtain semiperfect colorings of patterns, we use the framework by De Las Pe\~nas, Felix, and Quilinguin  in \cite{DLPFQ1} and \cite{DLPFQ2}.  Let $H\leq G$ and $Y$ a complete set of right coset representatives of $H$ in $G$.  A $(Y_i,J_i)-H$ \emph{partition of} $G$ is the partition $\{hJ_iY_i:i\in I, h\in H\}$ of $G$ where $Y=\cup_{i\in I}{Y_i}$ and $J_i\leq H$ $\forall i\in I$.  If $P$ is a $(Y_i,J_i)-H$ partition of $G$, then $hP=P$ $\forall h\in H$.  That is, all the elements of $H$ will permute the colors in a coloring associated with a $(Y_i,J_i)-H$ partition of $G$.

Using this framework, we may obtain semiperfect colorings of symmetrical patterns by observing the following procedure:
\begin{enumerate}
	\item Choose a subgroup $H$ of index 2 in $G$.

	\item Choose a complete set $Y$ of right coset representatives of $H$ in $G$.

	\item Partition $Y$ either as $Y=Y_1$ or as $Y=Y_1\cup Y_2$.

	\item \begin{itemize}
		\item If $Y=Y_1$, choose $J_1\leq H$ and form the partition $\{hJ_1Y_1:h\in H\}$.  Such partitions will be referred to as 
		\emph{Type I} partitions and they give rise to colorings having only one orbit of colors.

		\item If $Y=Y_1\cup Y_2$, choose $J_1$, $J_2\leq H$ and form the partition $\{hJ_1Y_1:h\in H\}\cup\{hJ_2Y_2:h\in H\}$.  Such 
		partitions will be referred to as \emph{Type II} partitions and they give rise to colorings having at most two orbits of colors.
	\end{itemize}
\end{enumerate}

To illustrate, consider the uncolored hexagonal pattern in Figure \ref{fig1}(a). Its symmetry group is $G=\langle a,b\rangle\cong D_6$ where $a$ is the $60^{\circ}$-counterclockwise rotation about the center of the hexagon and $b$ is the reflection along the horizontal line through the center of the hexagon.  Note that the pattern may be obtained as the $G$-orbit of the tile labeled ``$e$'' and we obtain an assignment of each element of $G$ to a unique tile in the pattern as shown in Figure \ref{fig1}(a).

Let $H=\langle a^2,b\rangle$, a subgroup of index 2 in $G$, and $Y=\{e,a^3\}$, a complete set of right coset representatives of $H$ in $G$.  Write $Y=Y_1\cup Y_2$ where $Y_1=\{e\}$  and $Y_2=\{a^3\}$ and choose the subgroups $J_1=\langle a^2b\rangle$ and $J_2=H$ of $H$.  Form the Type II partition $\{hJ_1Y_1:h\in H\}\cup\{hJ_2Y_2:h\in H\}=\{\{e,a^2b\},\{b,a^4\},\{a^4b,a^2\},\{ab, a, a^3b, a^3, a^5b, a^5\}\}$.  Upon assigning the color yellow to $\{e,a^2b\}$, the color green to $\{b,a^4\}$, the color blue to $\{a^4b,a^2\}$, and the color red to $\{ab, a, a^3b, a^3, a^5b,a^5\}$, the coloring in Figure \ref{fig1}(b) is obtained.

In Figure \ref{fig1}(b), we see that the reflection $b$ fixes the colors red and blue (that is, all red tiles are mapped by $b$ to red tiles, and the same is true for blue tiles) and interchanges the colors green and yellow (that is, $b$ sends green tiles to yellow tiles and vice-versa).  Thus, $b$ permutes the colors in Figure \ref{fig1}(b).  On the other hand, the rotation $a$ does not permute the colors.  Indeed, some red tiles are mapped by $a$ to yellow tiles while others are mapped to blue (and green) tiles.  Looking at the effect of the other symmetries in $G$ on the colors, we conclude that the coloring in Figure \ref{fig1}(b) is semiperfect since only the elements of $H$ permute the colors.  Also, the coloring has two orbits of colors, namely, $\{\text{blue},\text{green},\text{yellow}\}$ and $\{\text{red}\}$.

\noindent\begin{figurehere}
\centering{\includegraphics[width=40mm]{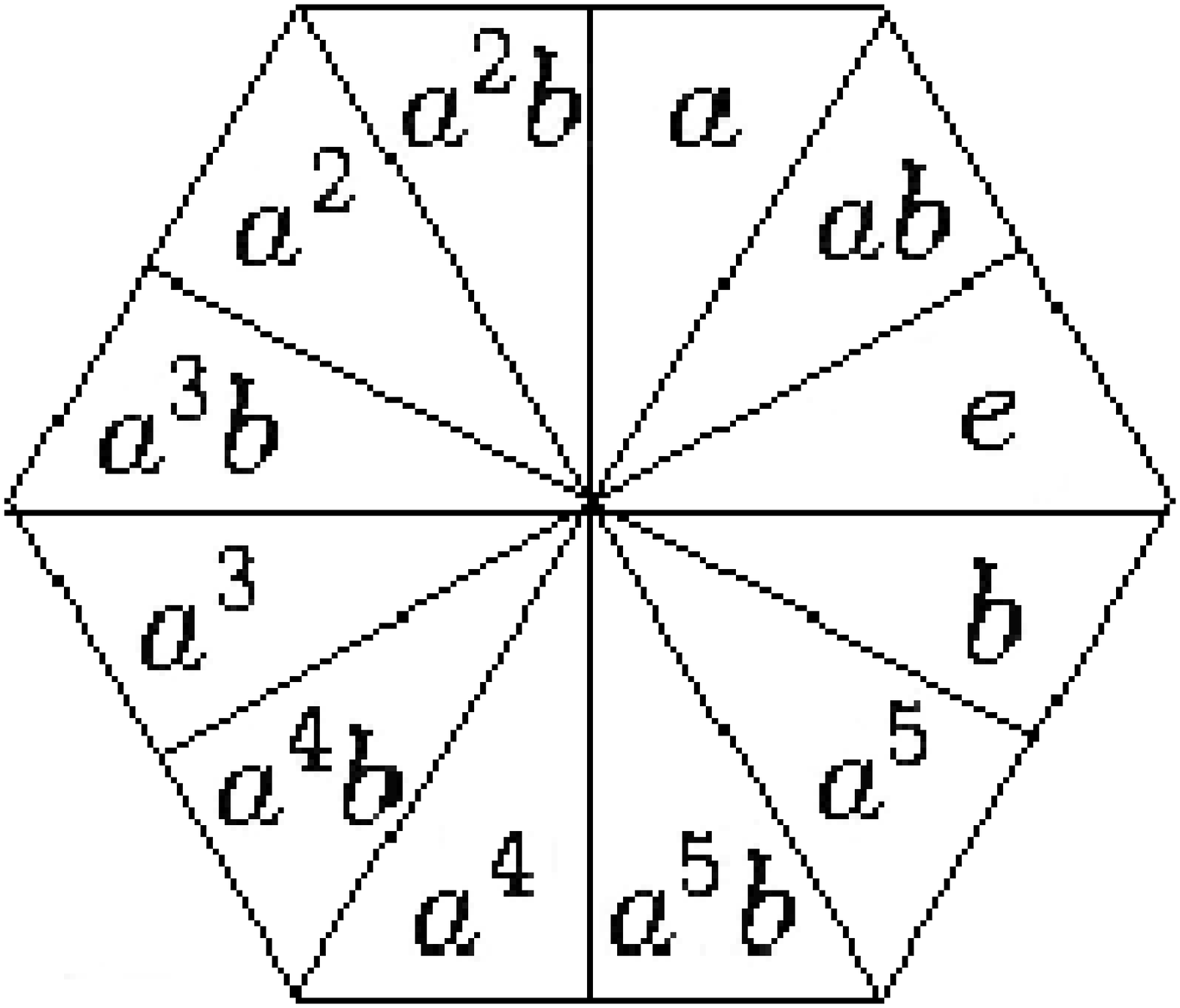}}

\centering{(a)}

	\begin{minipage}[c]{40mm}
		\centering{\includegraphics[width=35mm]{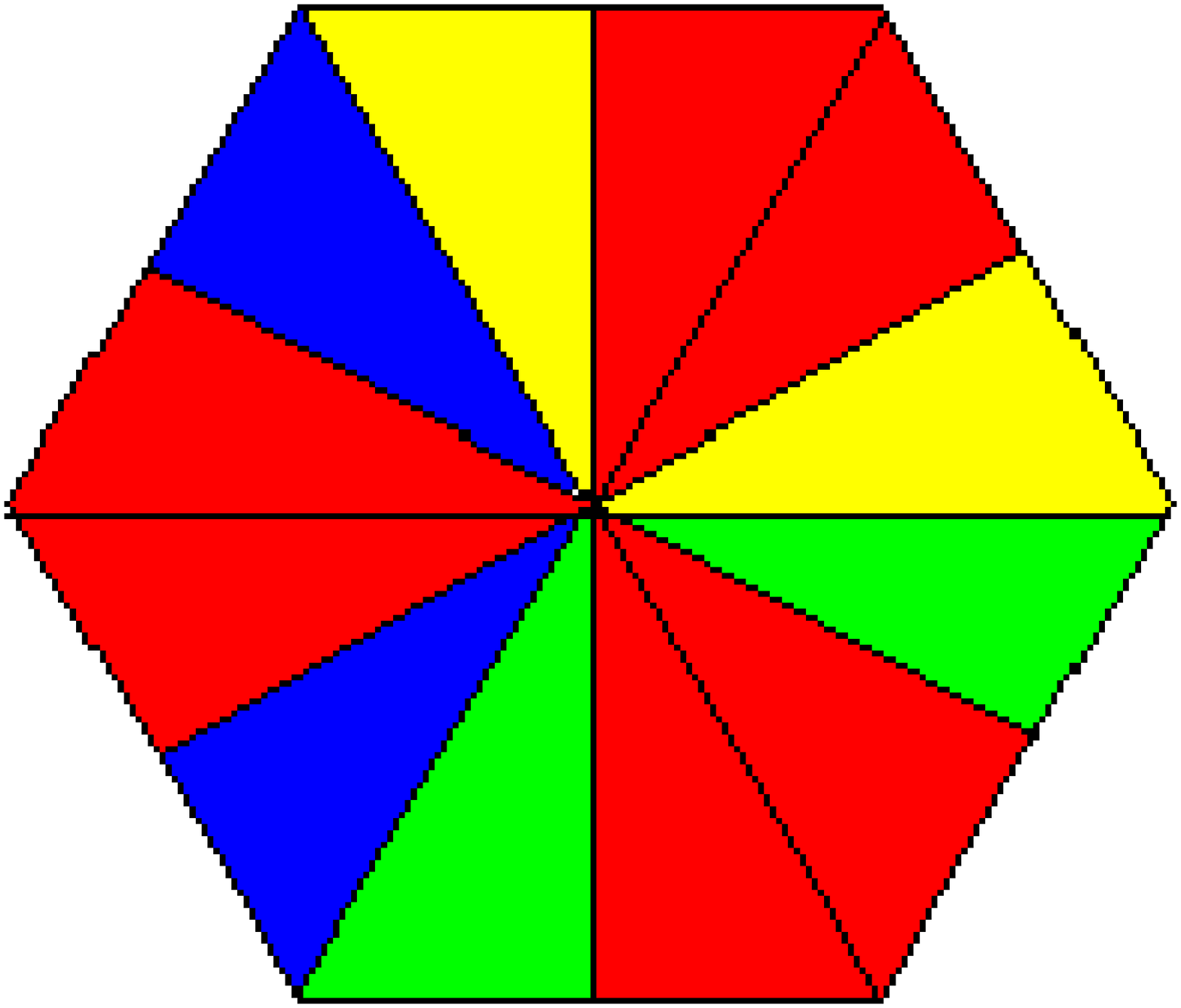}}
		
		\centering{(b)}
	\end{minipage}
	\begin{minipage}[c]{40mm}
		\centering{\includegraphics[width=35mm]{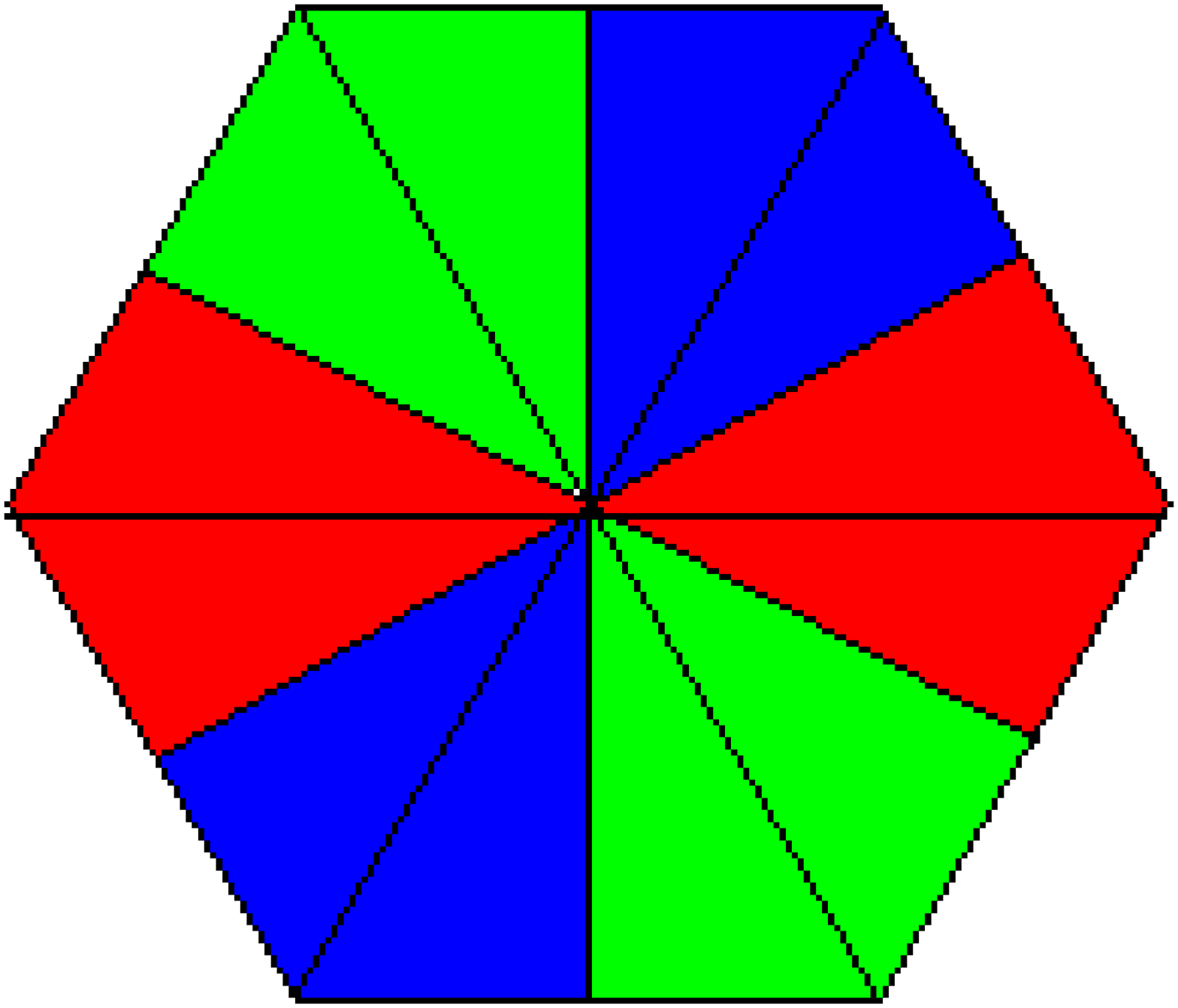}}

		\centering{(c)}
	\end{minipage}
\caption{Semiperfect (b) and perfect (c) colorings of the hexagonal pattern in (a) }
\label{fig1}
\end{figurehere}

\medskip If $H^*$ is the stabilizer of a $(Y_i,J_i)-H$ partition of $G$ then $H\leq H^*$.  Since we chose $H$ such that $[G:H]=2$, the coloring corresponding to the $(Y_i,J_i)-H$ partition of $G$ will be either semiperfect or perfect.  For instance, if we partition $Y$ in the previous example as $Y=\{e,a^3\}$ and choose $J=\langle b\rangle$, we generate the perfect coloring in Figure \ref{fig1}(c).  In this paper, we give methods on how to determine whether a $(Y_i,J_i)-H$ partition of $G$ where $[G:H]=2$ corresponds to a perfect or semiperfect coloring without doing the actual coloring.

\section{Equivalence of colorings}

In enumerating semiperfect colorings, we will only consider those that are inequivalent to each other.  Two colorings of the same symmetrical pattern
are said to be \emph{equivalent} \cite{Ro1} if one of the colored patterns may be obtained from the other colored pattern by

\begin{enumerate}
	\item a bijection from $C_1$ to $C_2$ where $C_i$ is the set of colors in the $i$th colored pattern for $i=1,2$, or

	\item a symmetry of the uncolored pattern, or

	\item a combination of $(1)$ and $(2)$.
\end{enumerate}

When we apply a symmetry of an uncolored pattern to a coloring of the pattern, we obtain a reassignment of the colors to different objects in the pattern.  Hence, only a relabelling of the colors is necessary to show that two perfect colorings are equivalent \cite{Ro1}.  However, given two
nonperfect colorings of the same pattern, it is usually not easy to determine whether they are equivalent. The following are some suggestions on
how to find out if the two colorings are equivalent or not.

\begin{enumerate}
	\item  Check that the number of colors used in both colorings are the same.  If this is not the case, then they must be inequivalent.

	\item  Count the number of color orbits formed in both colorings.  Note that the patterns formed by the colors belonging to one orbit of color 
	are necessarily congruent.  If the number of color orbits are not equal, then the colorings are inequivalent.

	\item  If the number of colors and the number of orbits of colors are the same, then using the same set of colors for both colorings 
	facilitates in distinguishing them.

	\item Choose a color in one of the colorings and identify the pattern formed by that color.  If the same pattern does not appear for some 
	color in the other coloring then the two colorings are inequivalent.
\end{enumerate}

For instance, the colorings in Figures \ref{fig1}(b) and \ref{fig1}(c) are clearly inequivalent because Figure \ref{fig1}(b) consists of four colors and two color orbits while Figure \ref{fig1}(c) consists of only three colors and one color orbit.

Consider the uncolored pattern in Figure \ref{fig2} which is assumed to repeat over the entire plane.  Its symmetry group is the crystallographic group $G=\langle a,b,x,y\rangle$ of type $p4m$ where $a$ is the $90^{\circ}$-counterclockwise rotation about the indicated point $P$, $b$ is the reflection along the horizontal line passing through $P$, and $x$ and $y$ are translations as indicated.

\noindent\begin{figurehere}
	\centering{\includegraphics[width=70mm]{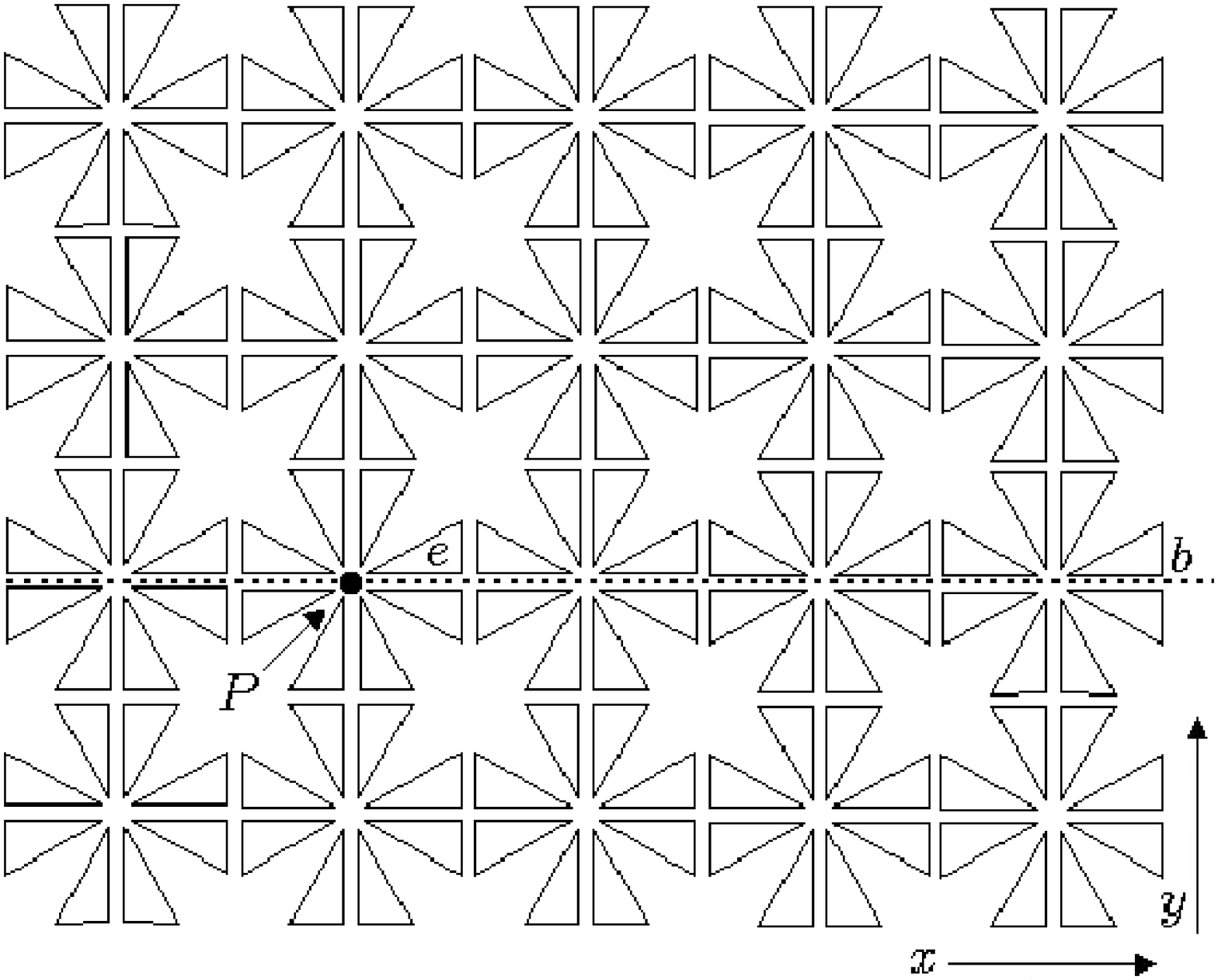}}
	
\caption{Uncolored pattern with symmetry group of type $p4m$}
\label{fig2}
\end{figurehere}

\noindent\begin{center}\begin{figurehere}
	\begin{minipage}[c]{50mm}
		\includegraphics[width=50mm]{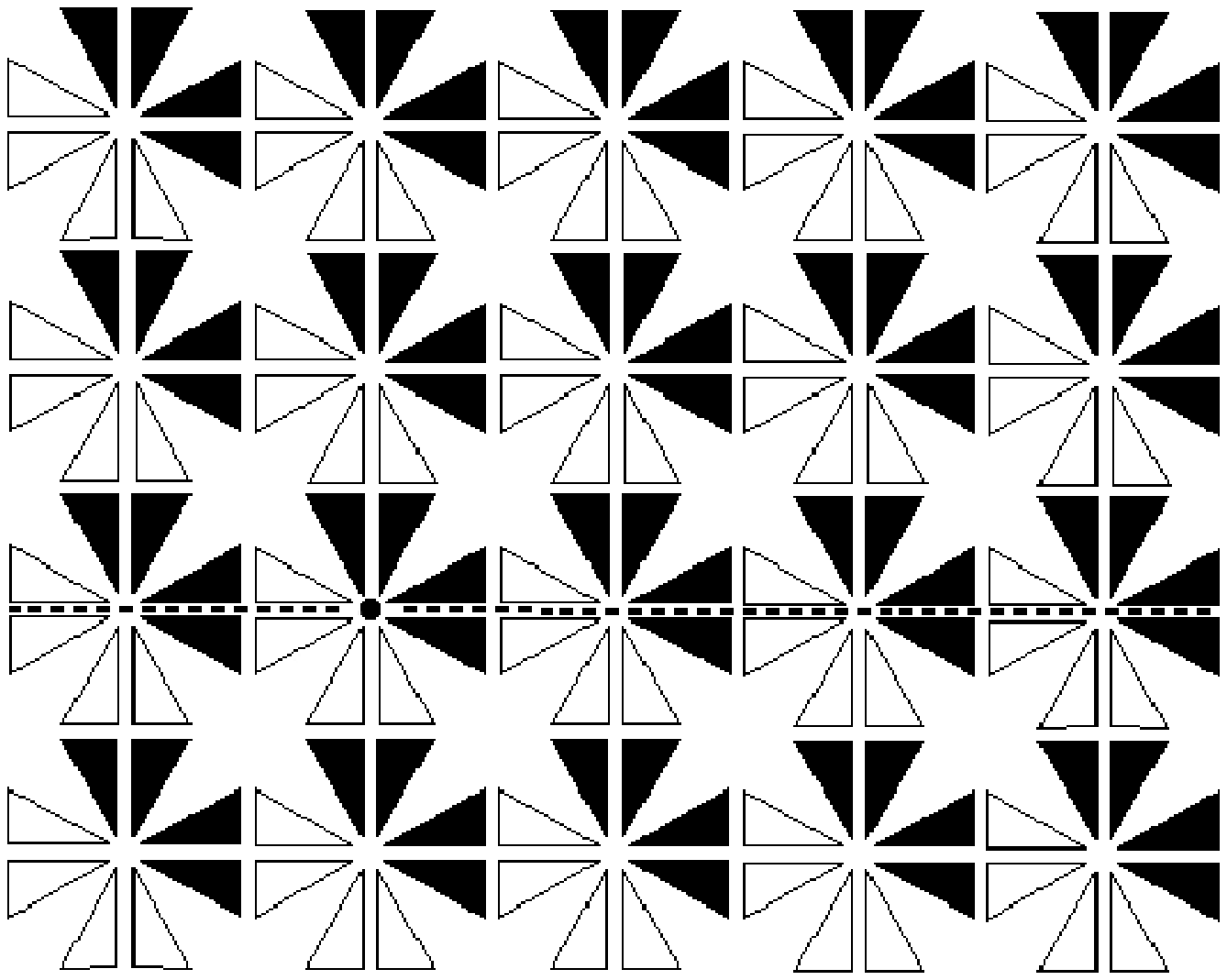}
		
		\centering{(a)}
	\end{minipage}\quad
	\begin{minipage}[c]{50mm}
		\includegraphics[width=50mm]{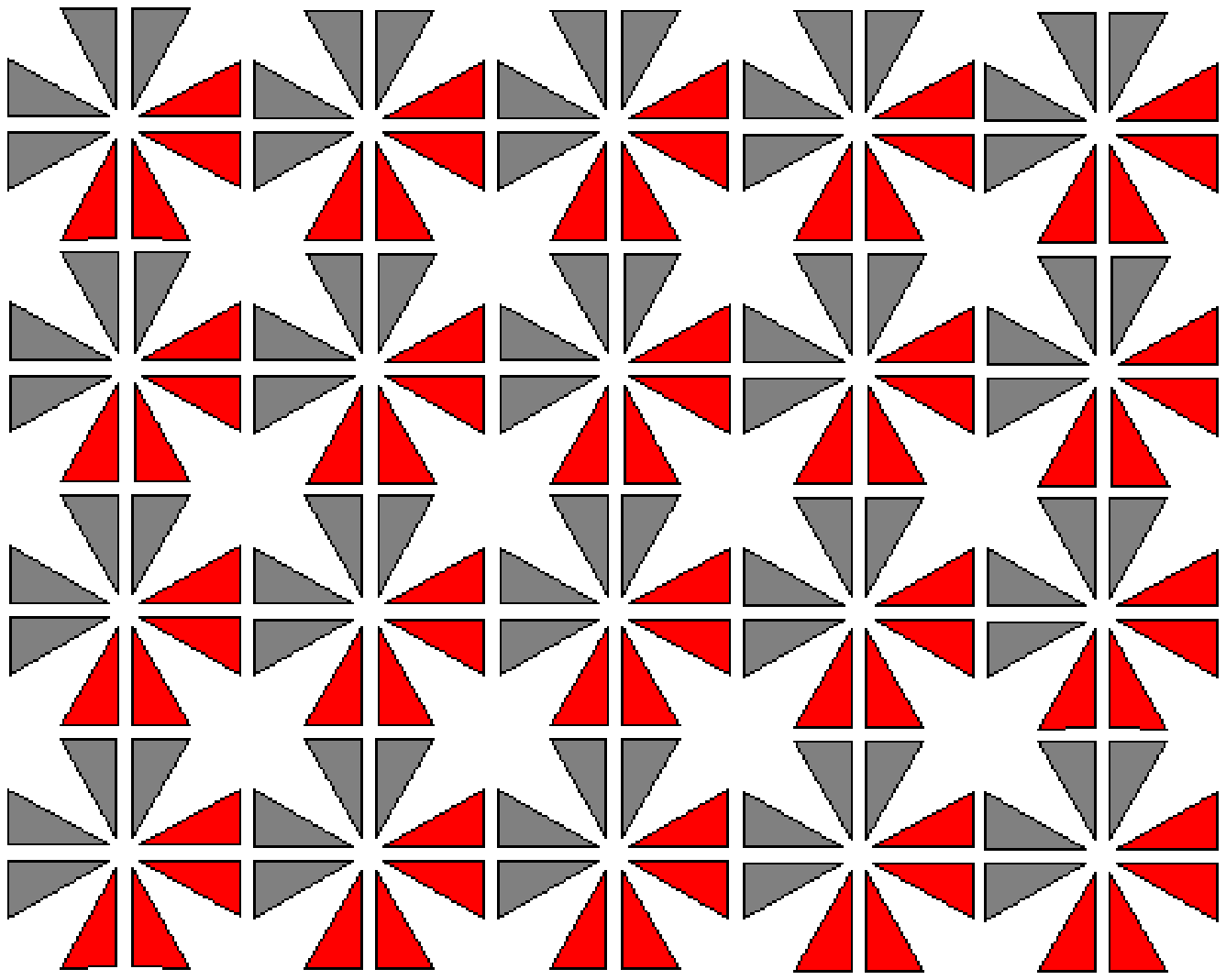}

		\centering{(b)}
	\end{minipage}
\caption{Two equivalent colorings of the uncolored pattern in Figure \ref{fig2}}
\label{fig3}
\end{figurehere}\end{center}

\medskip The colorings of the pattern in Figure \ref{fig3} are equivalent.  Indeed, suppose we replace the color black by red and the color white by gray in Figure \ref{fig3}(a).  We now see that the coloring in Figure \ref{fig3}(b) is just the image of Figure \ref{fig3}(a) (now colored also with red and gray) under the reflection $b$.

We now define equivalence of colorings via partitions of the symmetry group $G$ of the uncolored pattern.  If $P$ and $Q$ are partitions of $G$ and there is a $g\in G$ such that $Q=gP$, then we say that $P$ and $Q$, as well as their associated colorings, are \emph{equivalent}.

To illustrate, the partitions of $G=\langle a,b,x,y\rangle$ of type $p4m$ corresponding to the colorings in Figure \ref{fig3}(a) and (b) are $P=\{h\langle ab,x,y\rangle\{e,a\}:h\in\langle ab,a^3b,x,y\rangle\}$ and $Q=\{h\langle a^3b,x,y\rangle\{e,b\}:h\in\langle ab,a^3b,x,y\rangle\}$, respectively.  We have $Q=bP$ and hence, the partitions $P$ and $Q$, and the two colorings, are equivalent.

We will look at results on how one should select the subgroup $H$ of index 2 in $G$, the complete set of right coset representatives $Y$ of $H$ in $G$, the partition of $Y$, and the subgroups $J$, or $J_1$ and $J_2$ of $H$, in order to avoid obtaining colorings equivalent to those already generated. We start with the following theorem which is a special case of results in \cite{Q} and \cite{Ro1}.

\begin{theorem}\label{thm1}
 	Let $G$ be a group and $P$ an $H$-invariant partition of $G$.  If $[G:H]=2$ then there are only two partitions of $G$ that are equivalent to 
	$P$, namely, $P$ and $yP$, for some $y\in G\setminus H$.  Moreover, the stabilizer of $yP$ in $G$ is also $H$.
\end{theorem}
\begin{proof}
 	Let $E=\{gP:g\in G\}$ be the set of partitions of $G$ that are equivalent to $P$.  Then $E$ is the orbit of $P$ under the action of $G$ on 
	the set of partitions of $G$ by left multiplication.  By the Orbit-Stabilizer Theorem, $|E|=[G:H]=2$ and $E=\{P,yP\}$, for some $y\in 
	G\setminus H$.  Now, $g(yP)=yP\Leftrightarrow (y^{-1}gy)P=P\Leftrightarrow g\in yHy^{-1}=H$ since $H\unlhd G$.
\end{proof}

\section{Type II partitions}
We first consider Type II partitions of $G$ for a fixed subgroup $H$ of index 2 in $G$.  We want to enumerate all inequivalent partitions of $G$ of the form $P=\{hJ_1\{x\}:h\in H\}\cup\{hJ_2\{y\}:h\in H\}$ where $J_1$, $J_2\leq H$, $x\in H$, and $y\in G\setminus H$.  However, we can write
$P$ as $P=\{hJ_1':h\in H\}\cup y\{hJ_2':h\in H\}$ where $J_1'={J_1}^{x^{-1}}=x^{-1}J_1x\leq H$, $J_2'={J_2}^{\,y^{-1}}\leq H$.  Thus, we can look at instead partitions of $G$ of the form $\{hJ_1:h\in H\}\cup y\{hJ_2:h\in H\}$, where $J_1$, $J_2\leq H$, and $y\in G\setminus H$.

\begin{theorem}\label{thm2}
 	Let $H\leq G$ with $[G:H]=2$, $J_1$, $J_2\leq H$, $y\in G\setminus H$, and $P=P_1\cup yP_2$ where $P_i=\{hJ_i:h\in H\}$ for $i=1,2$.
	\begin{itemize}
	 	\item For all $y'\in G\setminus H$, $P_1\cup y'P_2=P_1\cup yP_2=P$.

		\item The two distinct partitions of $G$ equivalent to $P$ are $P$ and $P_2\cup yP_1=\{hJ_2:h\in H\}\cup y\{hJ_1:h\in H\}$.

		\item A coloring induced by $P$ is perfect if and only if $J_1=J_2$.
	\end{itemize}
\end{theorem}
\begin{proof}
	\begin{itemize}
		\item[]

		\item Since $y'\in G\setminus H=yH$, $y'P_2=yP_2$.

		\item By Theorem \ref{thm1}, $P$ is only equivalent to $P$ and $yP=y(P_1\cup yP_2)=P_2\cup yP_1$.

		\item The coloring corresponding to $P$ is perfect $\Leftrightarrow$ $gP=P$ $\forall g\in G$.  Since $hP=P$ $\forall h\in H$, and 
		$P_1$ and $P_2$ are partitions of $H$ into left cosets, we conclude that the coloring is perfect $\Leftrightarrow yP=P 
		\Leftrightarrow P_2\cup yP_1=P_1\cup yP_2\Leftrightarrow P_1=P_2\Leftrightarrow J_1=J_2$. \qedhere
	\end{itemize}
\end{proof}	

Theorem \ref{thm2} suggests that in order to generate all inequivalent Type II $H$-invariant partitions of $G$ that correspond to semiperfect colorings;
\begin{enumerate}
	\item Choose any $y\in G\setminus H$.

	\item Take all possible 2-combinations $\{J_1,J_2\}$ of subgroups of $H$.

	\item Form the different partitions $\{hJ_1:h\in H\}\cup y\{hJ_2:h\in H\}$ of $G$.
\end{enumerate}
Therefore, if $H$ has $n$ subgroups then there are $\binom{n}{2}$ inequivalent $H$-invariant partitions of $G$ of Type II.

Consider the subgroup $H=\langle a^2,b\rangle$ of index 2 in the symmetry group $G=\langle a,b\rangle$ of the uncolored hexagonal pattern in Figure
\ref{fig1}(a).  Since $H$ has 6 subgroups, there are $\binom{6}{2}=15$ inequivalent semiperfect colorings of the hexagonal pattern with $H=\langle a^2,b\rangle$ as its associated color group and two orbits of colors.

Remember that the uncolored repeating pattern in Figure \ref{fig2} has symmetry group $G$ of type $p4m$.  Choose $H$ to be a subgroup of index 2 in $G$ that is also of type $p4m$.  The group $H$ has 7 subgroups of index 2 and no subgroups of index 3.  Hence, there are $\binom{8}{2}+0=28$ inequivalent Type II $H$-invariant partitions of $G$ that correspond to colorings having at most four colors.

We remark here that Theorem \ref{thm2}(c) is equivalent to the result in \cite{DLPF} (see also \cite{DLPP}) that the Type II partition of $G$ of the form $\{hJ_1:h\in H\}\cup\{hJ_2\{y\}:h\in H\}$ corresponds to a perfect coloring if and only if $J_2={J_1}^{y}$.  For example, recall that the coloring in Figure \ref{fig1}(b) was obtained from a Type II partition of $G=\langle a,b\rangle$ for which $J_1=\langle a^2b\rangle$, $J_2=H=\langle a^2,b\rangle$, $Y_1=\{e\}$, and $Y_2=\{a^3\}$.  Clearly, $J_2\neq {J_1}^{a^3}$ and this is why the coloring in Figure \ref{fig1}(b) is semiperfect.

\section{Type I partitions}

We now turn our attention to partitions of $G$ of Type I.

\begin{lemma}\label{lemma3}
	Let $H\leq G$ with $[G:H]=2$.  Let $J\leq H$, $g\in G$, $Y$ a complete set of right coset representatives of $H$ in $G$, and 
	$g^{-1}Y=\{x,y\}$ where $x\in H$, $y\in G\setminus H$.  The following statements hold:
	\begin{itemize}
	 	\item $\{hJ^gY:h\in H\}=g\{hJ(g^{-1}Y):h\in H\}$

		\item $\{hJ\{x,y\}:h\in H\}=x\{hJ^{x^{-1}}\{e,x^{-1}y\}:h\in H\}$

		\item If $J'=J^{x^{-1}}$ and $y'\in J'x^{-1}y$, then

		\item[] $\{hJ'\{e,x^{-1}y\}:h\in H\}=\{hJ'\{e,y'\}:h\in H\}$.

		\item $\{hJ'\{e,y'\}:h\in H\}=\{h(J'\cup J'y'):h\in H\}$
	\end{itemize}
\end{lemma}

For a given subgroup $H$ of index 2 in $G$, we want to determine all the inequivalent Type I partitions of $G$ that could give rise 
to semiperfect colorings.  It is enough to consider a representative $J$ from each conjugacy class of subgroups of $H$ in $G$ by Lemma \ref{lemma3}(a).
From Lemma \ref{lemma3}(b), the complete set of right coset representatives of $H$ in $G$ may be taken as $\{e,y\}$ as long as we consider the conjugates of each $J$ by elements of $H$.  This, together with Lemma \ref{lemma3}(d), implies that we only need to look at Type $I$ partitions of $G$ of the form $P=\{h(J\cup Jy):h\in H\}$.  Note that $P$ is determined by the set $J\cup Jy$ where $J\leq H$ and $y\in G\setminus H$.  In addition, Lemma \ref{lemma3}(c) indicates that when choosing $y$, we only need to take representatives from each right coset of $J$ in $G$ that is contained in $G\setminus H$.

Assume that a representative subgroup $J$ from each conjugacy class of subgroups of $H$ in $G$ has been chosen.  From Lemma \ref{lemma3}, we only need to look at all partitions of the form $P^l(r):=\{h(J^l\cup J^lr):h\in H\}$ where $l$ runs over a complete set of left coset representatives of $N_H(J)$ in $H$ and $r$ runs over a complete set of right coset representatives of $J^l$ in $G$ that are not in $H$.  Note that there are $[H:N_H(J)]\cdot[H:J]$ such partitions.  However, some of these partitions may still be equivalent to each other.  We address this problem in the next theorem.

\begin{theorem}\label{thm4}
 	Suppose $J\leq H\leq G$ with $[G:H]=2$.  Let $L$ be a complete set of left coset representatives of $N_H(J)$ in $H$ and for each $l\in L$,
	let $R(l)$ be a complete set of right coset representatives of $J^l$ in $G$ that are not in $H$.  If $N_G(J)=N_H(J)$ then the semiperfect
	colorings associated with the partitions $P^l(r)=\{h(J^l\cup J^lr):h\in H\}$ of $G$ for each $l\in L$, $r\in R(l)$ are inequivalent to each 
	other.  Otherwise, a pairing of equivalent semiperfect colorings is obtained.
\end{theorem}
\begin{proof}
 	By Theorem \ref{thm1}, if $P^l(r)$ corresponds to a semiperfect coloring then a partition $P^x(y)$ is equivalent to $P^l(r)$ if and only if 
	$P^x(y)=P^l(r)$ or $P^x(y)=rP^l(r)$.  We have $P^x(y)=P^l(r) \Leftrightarrow x=l$ and $y=r$.  Now,
	\begin{alignat*}{2}
	 	rP^l(r)&=r^{-1}\{h(J^l\cup J^lr):h\in H\}\\
		&=\{h(r^{-1}J^l\cup r^{-1}J^lr):h\in H\}\\
		&=\{h(J^{r^{-1}l}\cup J^{r^{-1}l}r^{-1}):h\in H\}.
	\end{alignat*}
	Thus, $P^x(y)=rP^l(r)\Leftrightarrow J^x\cup J^xy=J^{r^{-1}l}\cup J^{r^{-1}l}r^{-1}\Leftrightarrow J^x=J^{r^{-1}l}$ and $y=r^{-1}$.  Note 
	that $J^x=J^{r^{-1}l} \Leftrightarrow J=J^{x^{-1}r^{-1}l}$ and this cannot happen if $N_G(J)=N_H(J)$ because $x^{-1}r^{-1}l\in G\setminus H$.
	Hence, none of the partitions $P^l(r)$ are equivalent to each other when $N_G(J)=N_H(J)$.  On the other hand, suppose  $N_G(J)\neq N_H(J)$, 
	that is, there exists $g\in G\setminus H$ with $J^g=J$.  We have $r^{-1}lg\in H$ and $J^{r^{-1}lg}={(J^g)}^{r^{-1}l}=J^{r^{-1}l}$.	 
	Therefore, $P^l(r)$ is equivalent to $P^x(y)$ where $x=r^{-1}lg$ and $y=r^{-1}$ and when we consider all the colorings corresponding to the 
	partitions of the form $P^l(r)$, we get a pairing of equivalent semiperfect colorings.
 \end{proof}

Note that if $R=\{r_i\}$ is a complete set of right coset representatives of $J$ in $G$ that are not in $H$ then for all $l\in L$, $R^l=\{{r_i}^l\}$ is a complete set of right coset representatives of $J^l$ in $G$ that are not in $H$.  Hence, for the rest of this section, we will simply take $R(l)$ to be $R^l$ for each $l\in L$.

Recall that the colorings corresponding to partitions of $G$ of the form $P^l(r^l)$ are either perfect or semiperfect.  The next theorem says when each of these situations occur.

\begin{theorem}\label{thm5}
 	Let $G$ be a group, $H$ a subgroup of $G$ of index 2, $J\leq H$, $l\in H$, and $r\in G\setminus H$.
	\begin{itemize}
		\item The coloring associated with the partition $P=\{h(J\cup Jr):h\in H\}$ of $G$ is perfect if and only if $rJ=Jr$ (or $r\in 
		N_G(J)$) and $r^2\in J$.

		\item The coloring corresponding to the partition $P=\{h(J\cup Jr):h\in H\}$ of $G$ is perfect if and only if the coloring 
		corresponding to the partition $P^l(r^l)=\{h(J^l\cup J^lr^l):h\in H\}$ of $G$ is perfect.
	\end{itemize}
\end{theorem}
\begin{proof}
 	\begin{itemize}
 		\item[]

		\item Since $hP=P$ $\forall h\in H$, the coloring corresponding to $P$ is perfect if and only if $rP=P$ in which case $rJ\cup rJr=
		J\cup Jr$ or $rJ=Jr$ (both being subsets of $G\setminus H$) and $rJr=J$ (both being subsets of $H$).  Now, $rJ=Jr$ and 
		$rJr=J\Leftrightarrow rJ=Jr$ and $r^2J=J$ or $r^2\in J$.

		\item This follows from (a) since $rJ=Jr\Leftrightarrow r^lJ^l=J^lr^l$ and $r^2\in J\Leftrightarrow {(r^l)}^2\in J^l$.\qedhere
 	\end{itemize}
\end{proof}

Theorem \ref{thm5}(a) implies that the number of partitions $P=\{h(J\cup Jr):h\in H\}$ of $G$, for a fixed subgroup $J$ of $H$ and where $r$ runs over a complete set of right coset representatives of $J$ in $G$ that are not in $H$, that correspond to perfect colorings is \[p(J):=(\text{the number of involutions in }N_G(J)/J)-(\text{the number of involutions in }N_H(J)/J).\]  In addition, we have $p(J)=p(J^l)$ $\forall l\in H$ by Theorem \ref{thm5}(b).  Therefore, for each representative subgroup $J$ from each conjugacy class of subgroups of $H$ in $G$, the number of inequivalent partitions $P^l(r^l)$ of $G$ that give rise to semiperfect colorings is $[H:N_H(J)]\cdot [H:J]$ if $N_G(J)=N_H(J)$ and $\frac{1}{2}[H:N_H(J)]([H:J]-p(J))$ otherwise.

To illustrate, recall that the uncolored hexagonal pattern in Figure \ref{fig1}(a) has symmetry group $G=\langle a,b\rangle$.  Choose the subgroup $H=\langle a^2,b\rangle$ of index 2 in $G$. A complete set of representative subgroups $J$ from each conjugacy class of subgroups of $H$ in $G$ is $\{H,\langle a^2\rangle, \langle b\rangle, \{e\}\}$.  We need to consider the different partitions $P^l(r^l)$ for each representative $J$ and determine if the associated coloring for each partition  will be perfect or semiperfect. The computations and corresponding results are summarized in Table \ref{Table1}.

In particular, consider the case when we choose $J=\langle b\rangle$, $l=e$, and $r=a^3$.  The coloring in Figure \ref{fig1}(c) corresponds to the partition $\{h(J\cup Jr):h\in H\}$ of $G$ and is perfect because $a^3J=Ja^3$ and ${(a^3)}^2=e\in J$.  Observe also that we obtain a pairing of equivalent semiperfect colorings when $J=\langle b\rangle$ since $\langle a^3,b\rangle=N_G(J)\neq N_H(J)=\langle b\rangle$.

\begin{tablehere}
	\begin{center}
		\begin{tabular}{|c|c|c||c|}
			\hline
			$J$ & $l$ & $r^l$ & Resulting Coloring\\\hline\hline
			$H$ & $e$ & $a$ & perfect\\\hline
			$\langle a^2\rangle$ & $e$ & $a$ & perfect\\\cline{3-4}
			& & $ab$ & perfect\\\hline
			& & $a$ & (1) semiperfect\\\cline{3-4}
			& $e$ & $a^3$ & perfect\\\cline{3-4}
			& & $a^5$ & (2) semiperfect\\\cline{2-4}
			& & $a$ & (3) semiperfect\\\cline{3-4}
			$\langle b\rangle$ & $a^2$ & $a^3$ & perfect\\\cline{3-4}
			& & $a^5$ & equivalent to (1)\\\cline{2-4}
			& & $a$ & equivalent to (2)\\\cline{3-4}
			& $a^4$ & $a^3$ & perfect\\\cline{3-4}
			& & $a^5$ & equivalent to (3)\\\hline
			& & $ab$ & perfect\\\cline{3-4}
			& & $a$ & (4) semiperfect\\\cline{3-4}
			$\{e\}$ & $e$ & $a^3b$ & perfect\\\cline{3-4}
			& & $a^3$ & perfect\\\cline{3-4}
			& & $a^5b$ & perfect\\\cline{3-4}
			& & $a^5$ & equivalent to (4)\\\hline				
		\end{tabular}
	\end{center}
	\caption{Subgroups $J$, left coset representatives $l$ of $N_H(J)$ in $H=\langle a^2,b\rangle$, and right coset representatives $r^l$ of $J^l$ 
	in $G=\langle a,b\rangle$ that should be considered to enumerate all inequivalent Type I $H$-invariant partitions of $G$} 
	\label{Table1}
\end{tablehere}

\medskip
It is convenient to use Theorem \ref{thm5}(a) to determine whether a coloring corresponding to a Type I partition of $G$ is perfect or semiperfect especially when $G$ is infinite.  Consider again the uncolored infinite repeating pattern in Figure \ref{fig2} whose symmetry group is $G=\langle a,b,x,y\rangle$.  Suppose we choose the subgroup $H=\langle b,a^2b,x,y\rangle$ of type $pmm$ and of index 2 in $G$, the subgroup $J=\langle xa^2b,xy,xy^{-1}\rangle$ of type $cm$ in $H$, and $r=a\in G\setminus H$.  Since $J$ does not contain any rotations, $a^2\notin J$.  Hence, a coloring obtained from the partition $\{h(J\cup Jr):h\in H\}$ of $G$ must be semiperfect by Theorem \ref{thm5}(a).

Theorem \ref{thm5}(a) has many immediate consequences.  For instance, if $J\unlhd G$ then a coloring induced by the partition $P=\{h(J\cup Jr):h\in H\}$ of $G$ is perfect if and only if $r^2\in J$.  In particular, if $J=\{e\}$ then $P$ corresponds to a perfect coloring if and only if $r$ is an involution.  We also conclude that when enumerating inequivalent Type I $H$-invariant partitions of $G$, there is no need to consider $J=H$ since the partition obtained will always correspond to a perfect coloring.  Also, no Type I partition of the dihedral group $D_n=\langle a,b:a^n=b^2={(ab)}^2=e\rangle$ gives rise to chirally perfect colorings (that is, semiperfect colorings where the associated color group is $\langle a\rangle$).

We conclude this section with a useful geometric consequence of Theorem \ref{thm5}(a).  It makes use of the \emph{diagram}, $D(S)$, of a set $S$ of isometries in $\mathbf{R}^n$, which is the set of symmetry elements of the isometries in $S$.

\begin{corollary}\label{cor6}
	Let $H$ be a subgroup of index 2 in $G$, $J\leq H$, and $r\in G\setminus H$.   If $rD(J)\neq D(J)$ then a coloring induced by the partition 
	$\{h(J\cup Jr):h\in H\}$ of $G$ is semiperfect.
\end{corollary}
\begin{proof}
	We know that $D(rJr^{-1})=rD(J)$ (see \cite{CFF}).  Hence, if $rD(J)\neq D(J)$ then $rJ\neq Jr$.
\end{proof}

\begin{figurehere}
	\centering{\includegraphics[width=65mm]{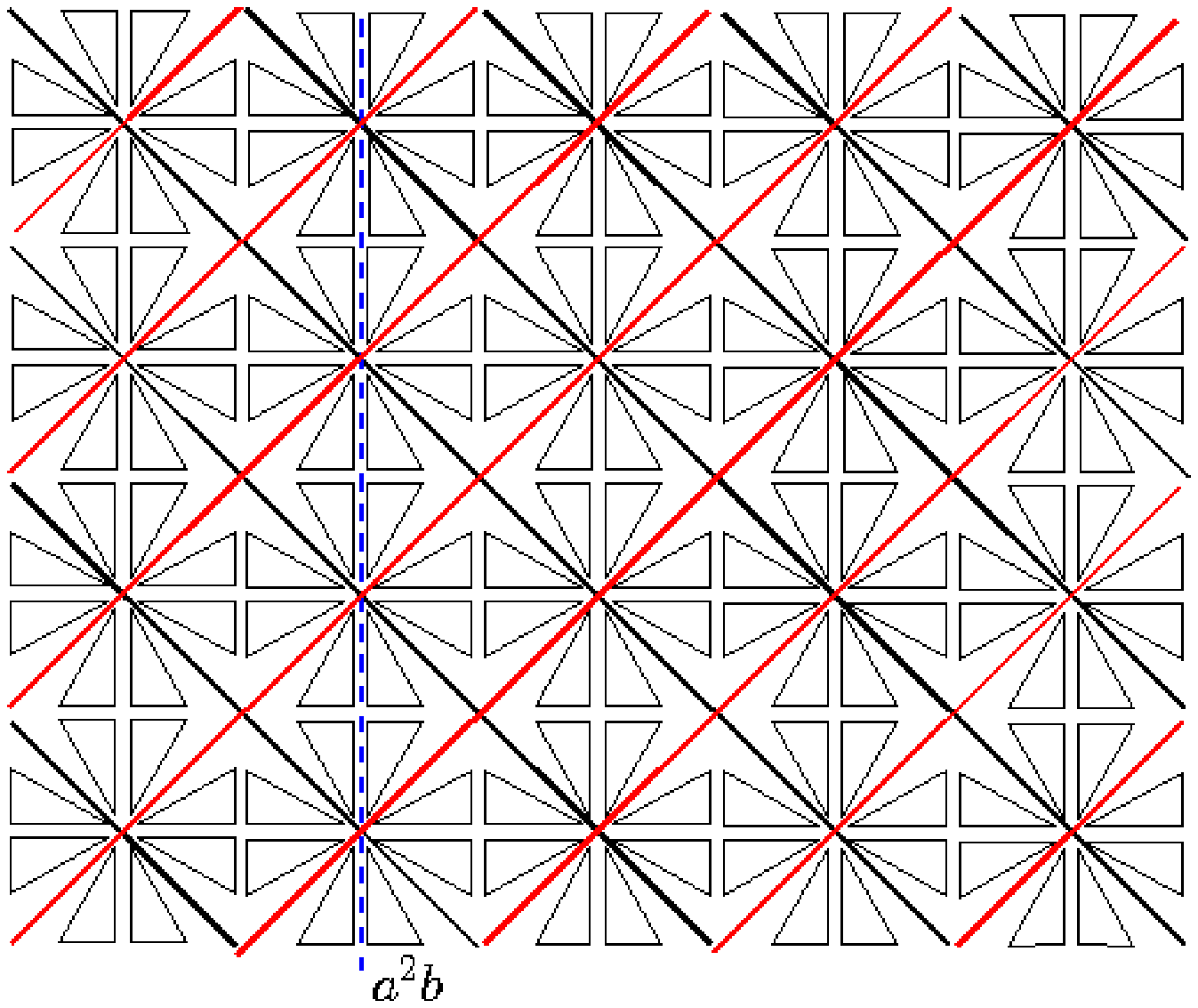}}
	
\caption{$D(J)$ and $(a^2b)D(J)$ where $J=\langle a^3b,xy,x^{-1}y\rangle$}
\label{fig4}
\end{figurehere}

\medskip For example, consider the subgroup $H=\langle xa,ab,xy,x^{-1}y\rangle$ of type $p4m$ and of index 2 in the symmetry group $G=\langle a,b,x,y\rangle$ of the pattern in Figure \ref{fig2}.  We choose $J=\langle a^3b,xy,x^{-1}y\rangle\leq H$ (of type $pm$) and $r=a^2b\in G\setminus H$.  Note that $a^2b$ is the reflection whose axis (dotted blue line) is shown in Figure \ref{fig4} and  ${(a^2b)}^2=e\in J$.  Now, $D(J)$ consists of the entire plane and the axes of reflections shown in black in Figure \ref{fig4}.  Hence, the image of $D(J)$ under $a^2b$ is the entire plane and the mirror axes shown in red in Figure \ref{fig4}.  We see that $D(J)\neq (a^2b)D(J)$ and so a coloring corresponding to the partition $\{h(J\cup Jr):h\in H\}$ of $G$ will be semiperfect by Corollary \ref{cor6}.

\section{Colorings associated with $(Y_i,J_i)-H$ partitions of $G$ for different subgroups $H$ of index 2 in $G$}

Let $H$ and $H'$ be two distinct subgroups of index 2 in $G$.  If $P$ and $P'$ are $H$- and $H'$-invariant partitions of $G$, respectively, then the colorings corresponding to $P$ and $P'$ must be inequivalent by Theorem \ref{thm1}.  Hence, we need to consider as $H$ all subgroups of index 2 in $G$ to enumerate all inequivalent partitions of $G$ that correspond to semiperfect colorings.

\noindent\begin{center}\begin{figurehere}
	\begin{minipage}[c]{40mm}
		\centering{\includegraphics[width=38mm]{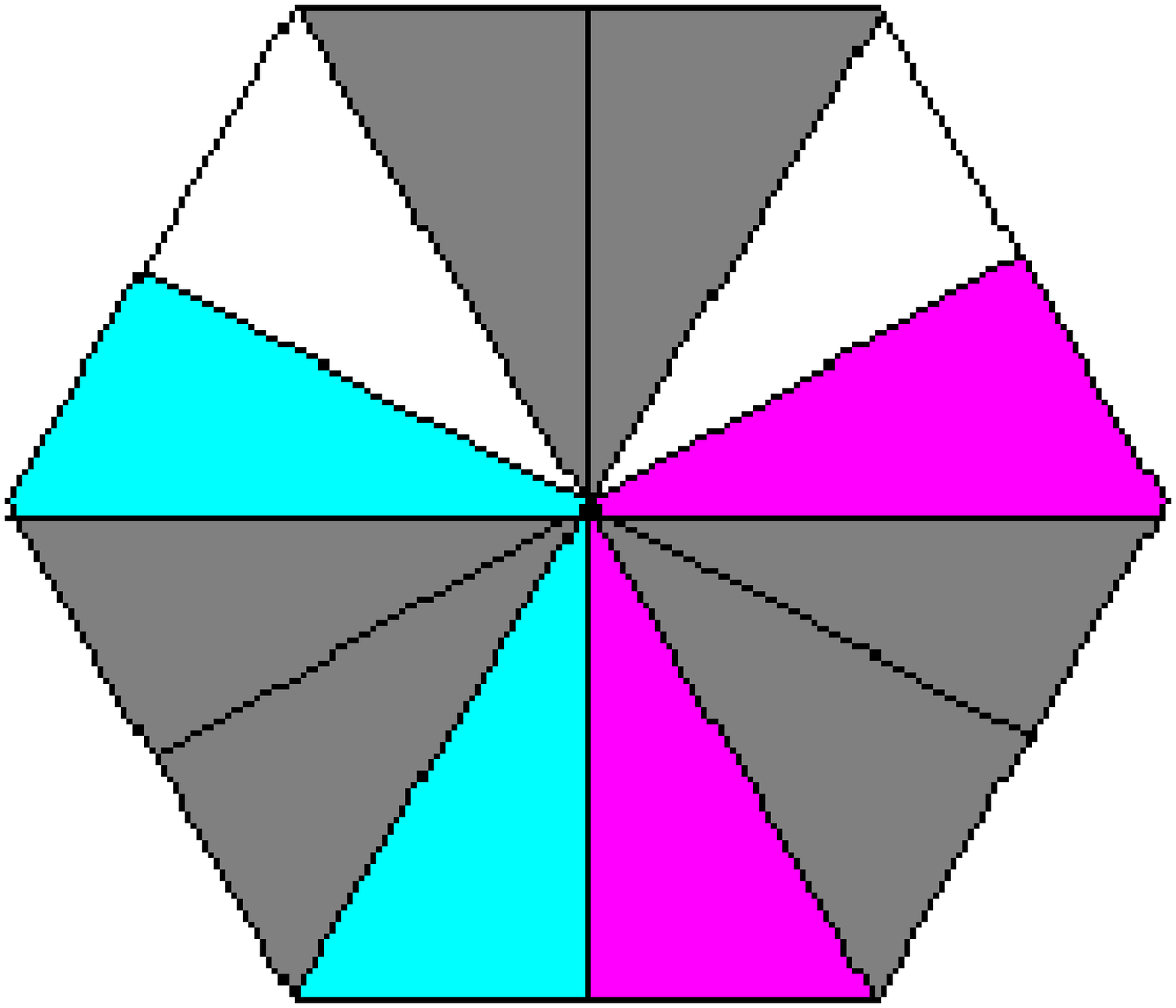}}
		
		\medskip\centering{(a)}
	\end{minipage}
	\begin{minipage}[c]{40mm}
		\centering{\includegraphics[width=38mm]{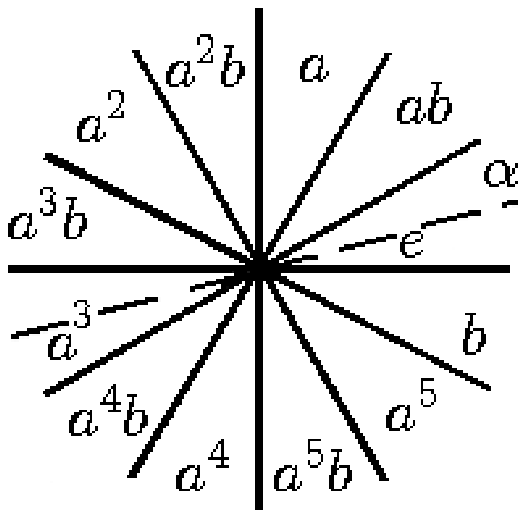}}

		\medskip\centering{(b)}
	\end{minipage}
\caption{(a) A semiperfect coloring inequivalent to Figure \ref{fig1}(b) ; (b) The reflection $\alpha\in N(G)$ and corresponding labels of fundamental domains of $G=\langle a,b\rangle$}
\label{fig5}
\end{figurehere}\end{center}

\medskip Recall that the color group associated to the semiperfect coloring in Figure \ref{fig1}(b) is $H=\langle a^2,b\rangle$.  The coloring of the same hexagonal pattern in Figure \ref{fig5}(a) is also semiperfect since its associated color group $H'=\langle a^2,ab\rangle$ is of index 2 in $G$.  Since the associated color groups of the semiperfect colorings in Figure \ref{fig1}(b) and Figure \ref{fig4}(a) are different, they are not equivalent by Theorem \ref{thm1}.  This is also evident geometrically because the color patterns in the two colorings are not congruent.

However, if one looks closely, even though the colorings in Figures \ref{fig1}(b) and \ref{fig5}(a) are inequivalent, they seem to be ``similar'' in some respects.  This phenomenon was noted by Macdonald and Street in \cite{MacS} and was considered by Roth in \cite{Ro1}, \cite{Ro2}, and \cite{Ro3}.

We will show that even if the colorings in Figures \ref{fig1}(b) and \ref{fig5}(a) are inequivalent, one may be obtained from the other by considering the images of the fundamental domains of $G$ under some element $\alpha$ of the normalizer of $G$, $N(G)$, in the group of isometries of $\mathbf{R}^n$.  To achieve this, we make use of the next result.

\begin{theorem}\label{thm7}
	Let $H$, $H'$ be subgroups of index 2 in $G$ with $H'=H^{\alpha}$ for some $\alpha\in N(G)$.
	\begin{itemize}
		\item If $D$ is a fundamental domain of $G$ and $\alpha (D)=D$ then the action of $\alpha$ on
		$\{g(D):g\in G\}$ is equivalent to the action of $\alpha$ on $G$ by conjugation.
			
		\item If $P$ and $P'$ are $(Y_i,J_i)-H$ and $({Y_i}^{\alpha},{J_i}^{\alpha})-H'$ partitions of $G$, respectively,
		then $P'=P^{\alpha}$.  Moreover, a coloring corresponding to $P$ is perfect if and only if a coloring corresponding
		to $P'$ is perfect.

		\item If $C$ and $C'$ are the sets of colors in colorings corresponding to $(Y_i,J_i)-H$ and $({Y_i}^{\alpha},{J_i}^{\alpha})-H'$ 
		partitions of $G$, respectively, then the action of $H$ on $C$ is equivalent to the action of $H'$ on $C'$.
	\end{itemize}
\end{theorem}
\begin{proof}
	\begin{itemize}
		\item[]

		\item There is a bijection between $G$ and the set $\{g(D):g\in G\}$ given by $g\leftrightarrow g(D)$.  Hence,
		$g^{\alpha}\in G$ is in a one-to-one correspondence with $g^{\alpha}(D)=\alpha(g(D))$ since $\alpha$, and hence $\alpha^{-1}$,
		stabilizes $D$.

		\item Note that if $J_i\leq H$ then ${J_i}^{\alpha}\leq H'$ for $i\in I$.  Also, if $Y=\{e,y\}$ is a complete set of right coset 
		representatives of $H$ in $G$, then $Y'=\{e,y^{\alpha}\}$ is a complete set of right coset representatives of $H'$ in $G$.  Hence,
		we obtain
		\begin{alignat*}{2}
			P^{\alpha}&=\{(\alpha h\alpha^{-1})(\alpha J_i\alpha^{-1})\{e,\alpha y\alpha^{-1}\}:i\in I, h\in H\}\\
			&=\{h'{J_i}^{\alpha}\{e,y^{\alpha}\}:i\in I,h'\in H'\}\\
			&=P'.
		\end{alignat*}
		If $P$ is a Type I partition, then by Theorem \ref{thm5}(a), either both $P$ and $P'=P^{\alpha}$ correspond to perfect colorings or 
		both correspond to semiperfect colorings since $Jy=yJ \Leftrightarrow J^{\alpha}y^{\alpha}=y^{\alpha}J^{\alpha}$ and
		$y^2\in J\Leftrightarrow {(y^{\alpha})}^2\in J^{\alpha}$.  The same result holds even if $P$ is of Type II because 
		$J_2={J_1}^y\Leftrightarrow {J_2}^{\alpha}={({J_1}^{\alpha})}^{y^{\alpha}}$.

		\item The elements of $C$ and $C'$ are in a one-to-one correspondence with the elements of the partitions $P=\{hJ_iY_i:i\in I,h\in
		H\}$ and $P'=\{h'J_i'Y_i':i\in I, h'\in H'\}$ of $G$, respectively, where $J_i'={J_i}^{\alpha}$ and $Y_i'={Y_i}^{\alpha}$ for $i\in 
		I$.  Since $P'=P^{\alpha}$, there is a bijection $f$ from $P$ to $P'$ given by $f(hJ_iY_i):=\alpha(hJ_iY_i)\alpha^{-1}
		=h^{\alpha}J_i'Y_i'$.  Let $\phi$ be the isomorphism from $H$ to $H'$ given by $\phi(h)=h^{\alpha}$.  The groups $H$ and $H'$ act on 
		$P$ and $P'$, respectively, by left multiplication, and hence $\forall g\in H$, $\phi(g)\cdot 
		f(hJ_iY_i)=g^{\alpha}(h^{\alpha}J_i'Y_i')=f(g\cdot hJ_iY_i)$.\qedhere
	\end{itemize}
\end{proof}	

Suppose $H'=H^{\alpha}$ for some $\alpha\in N(G)$.  Theorem \ref{thm7}(b) tells us that $(Y_i',J_i')-H'$ partitions of $G$ can be obtained by conjugating the $(Y_i,J_i)-H$ partitions of $G$ by $\alpha$. If $\alpha$ is chosen such that it stabilizes a fundamental domain $D$ of $G$, then by Theorem \ref{thm7}(a), not only can we obtain the colorings corresponding to $(Y_i',J_i')-H'$ partitions of $G$ from the colorings corresponding to $(Y_i,J_i)-H$ partitions of $G$ analytically by looking at the partitions, but also geometrically from the colorings.  That is, if $\alpha(D)=D$, then a coloring associated to the $({Y_i}^{\alpha},{J_i}^{\alpha})-H'$ partition of $G$ may be obtained from a coloring associated to the $(Y_i,J_i)-H$ partition of $G$ in the following manner:  If $R$ is the object of the pattern contained in $D$ then for all $g\in G$,
\begin{enumerate}
	\item associate the color of $g(R)$, in a coloring corresponding to the $(Y_i,J_i)-H$ partition of $G$, to $g(D)$;

	\item determine $\alpha(g(D))$;

	\item designate the color of $g(D)$ as the new color of $\alpha(g(D))$;

	\item assign to each $g^{\alpha}(R)$ the new color of $\alpha(g(D))$ and we obtain a coloring corresponding to the 
	$({Y_i}^{\alpha},{J_i}^{\alpha})-H'$  
	partition of $G$.
\end{enumerate}
Note that by Theorem \ref{thm7}(b), the two colorings that we get are either both perfect or both semiperfect.

Finally, even if the colorings induced by the $(Y_i',J_i')-H'$ partitions of $G$ are not equivalent to any of the colorings induced by the $(Y_i,J_i)-H$ partitions of $G$, Theorem \ref{thm7}(c) tells us that the color permutation groups that we obtain for both sets of colorings are isomorphic.  We illustrate all of these in the following examples.

Consider again the uncolored hexagonal pattern in Figure \ref{fig1}(a) whose symmetry group is $G=\langle a,b\rangle$.  If we denote by $c$ the $30^{\circ}$-counterclockwise rotation about the center of the hexagon, then $N(G)=\langle c,b\rangle\cong D_{12}$.

Let $\alpha=cb\in N(G)$ be the reflection along the angle bisector of the angle formed by the axes of the reflections $b$ and $ab$ (see Figure \ref{fig5}(b)).  Observe that $\alpha$ stabilizes the fundamental domain labeled ``$e$'' in Figure \ref{fig5}(b).  Also, the subgroups $H=\langle a^2,b\rangle$ and $H'=\langle a^2, ab\rangle$ of index 2 in $G$ are conjugate subgroups in $N(G)$ with $H'=H^{\alpha}$.

Recall that we obtained the coloring in Figure \ref{fig1}(b) by considering the partition 
	\begin{alignat*}{2}
		P&=\{h\langle a^2b\rangle\{e\}:h\in H\}\cup\{hH\{a^3\}:h\in H\}\\
		&=\{\{e,a^2b\},\{b,a^4\},\{a^4b,a^2\},\{ab, a, a^3b,a^3, a^5b, a^5\}
	\end{alignat*}
of $G$.  Now, consider the partition 
	\begin{alignat*}{2}
		P'&=\{h'{\langle a^2b\rangle}^\alpha\{e\}:h'\in H'\}\cup\{h'H^{\alpha}\{{(a^3)}^{\alpha}\}:h'\in H'\}\\
		&=\{h'\langle a^5b\rangle\{e\}:h'\in H'\}\cup\{h'H'\{a^3\}:h'\in H'\}\\
		&=\{\{e,a^5b\},\{ab,a^2\},\{a^3b,a^4\},\{a, a^2b, a^3, a^4b, a^5, b\}\} 
	\end{alignat*}	
of $G$. We see that Figure \ref{fig5}(a) is a coloring associated to $P'$.  Hence, we are able to obtain the coloring in Figure \ref{fig5}(a) from the coloring in Figure \ref{fig1}(b) by computing their corresponding partitions of $G$, as stated in Theorem \ref{thm7}(b).  Moreover, both colorings are semiperfect.

If we consider the images of the fundamental domain labeled ``$e$'' by elements of $G$, then we get the labelling of the fundamental domains as shown in Figure \ref{fig5}(b).  We now discuss how we can transform the coloring in Figure \ref{fig1}(b) to the coloring in Figure \ref{fig5}(a).  We associate the colors in the coloring in Figure \ref{fig1}(b) to their corresponding fundamental domains.  Getting the image of each fundamental domain under $\alpha$, we obtain a new assignment of colors to the fundamental domains (see Table \ref{Table2}).  Associating the new colors designated to each fundamental domain to the corresponding tile of the pattern, and changing the color yellow to purple, green to white, blue to light blue, and red to gray, we then obtain the coloring in Figure \ref{fig5}(a).

\begin{tablehere}
	\begin{center}
		\begin{tabular}{|c||c|c|c|}
			\hline
			Fundamental & Original & Image & New\\
			domain & color & under $\alpha$ & color\\\hline\hline
			$e$ & yellow & $e$ & yellow\\\hline
			$ab$ & red & $b$ & green\\\hline
			$a$ & red & $a^5$ & red\\\hline
			$a^2b$ & yellow & $a^5b$ & red\\\hline
			$a^2$ & blue & $a^4$ & green\\\hline
			$a^3b$ & red & $a^4b$ & blue\\\hline
			$a^3$ & red & $a^3$ & red\\\hline
			$a^4b$ & blue & $a^3b$ & red\\\hline
			$a^4$ & green & $a^2$ & blue\\\hline
			$a^5b$ & red & $a^2b$ & yellow\\\hline
			$a^5$ & red & $a$ & red\\\hline
			$b$ & green & $ab$ &red\\\hline
		\end{tabular}
	\end{center}
	\caption{Original color, image under $\alpha$, and new color of the fundamental domains of $G=\langle a,b\rangle$}
	\label{Table2}
\end{tablehere}

\medskip Denote the colors blue, green, red, yellow by 1, 2, 3, 4, respectively, and the colors light blue, white, gray, purple by 1', 2', 3', 4', respectively.  Then the set of colors of the coloring in Figure \ref{fig1}(b) is $C=\{1,2,3,4\}$ and the set of colors of the coloring in Figure \ref{fig5}(a) is $C'=\{1',2',3',4'\}$.  Table \ref{Table3} gives the color permutations induced by the elements of $H$ and $H'$ on their corresponding colorings.  Clearly, the action of $H$ on $C$ is equivalent to the action of $H'$ on $C'$.

\begin{tablehere}
	\begin{center}
		\begin{tabular}{|c|c||c|c|}
 			\hline
			$h\in H$ & Color & $h^{\alpha}\in H'$ & Color \\
			& permutation && permutation\\\hline\hline
			$e$ & (1) & $e$ & (1')\\\hline
			$a^2$ & (124) & $a^4$ & (1'2'4')\\\hline
			$a^4$ & (142) & $a^2$ & (1'4'2')\\\hline
			$b$ & (24) & $ab$ & (2'4')\\\hline
			$a^2b$ & (12) & $a^5b$ & (1'2')\\\hline
			$a^4b$ & (14) & $a^3b$ & (1'4')\\\hline
		\end{tabular}
	\end{center}
	\caption{Color permutations corresponding to the action of the elements of $H$ and $H'$ on $C$ and $C'$, respectively}
	\label{Table3}
\end{tablehere}

\medskip We also give an example for the infinite repeating pattern in Figure \ref{fig2} whose symmetry group is $G=\langle a,b,x,y\rangle$. Figure \ref{fig6}(a) shows the fundamental domains of $G$.  Let $\alpha\in N(G)$ be the reflection along the broken line in Figure \ref{fig6}(a).  Note that $\alpha$ stabilizes the fundamental domain labeled ``$e$''.  The groups $H=\langle a,ab,xy,x^{-1}y\rangle$ and $H'=\langle
xa,ab,xy,x^{-1}y\rangle$ are both subgroups of index 2 in $G$ and of type $p4m$ with $H'=H^{\alpha}$.  Choose $J=\langle a^2b,b,x^2,y^2\rangle\leq H$ (of type $pmm$) and $Y=\{e,xab\}$.  Since ${(xab)}^2=xy\notin J$, the coloring induced by the $(Y,J)-H$ partition of $G$ in Figure \ref{fig6}(b) is semiperfect by Theorem \ref{thm5}(a).  A semiperfect coloring corresponding to the  $(Y^{\alpha},J^{\alpha})-H'$ coloring of $G$ is shown in Figure \ref{fig6}(c).

\noindent\begin{figurehere}
\centering{\includegraphics[width=70mm]{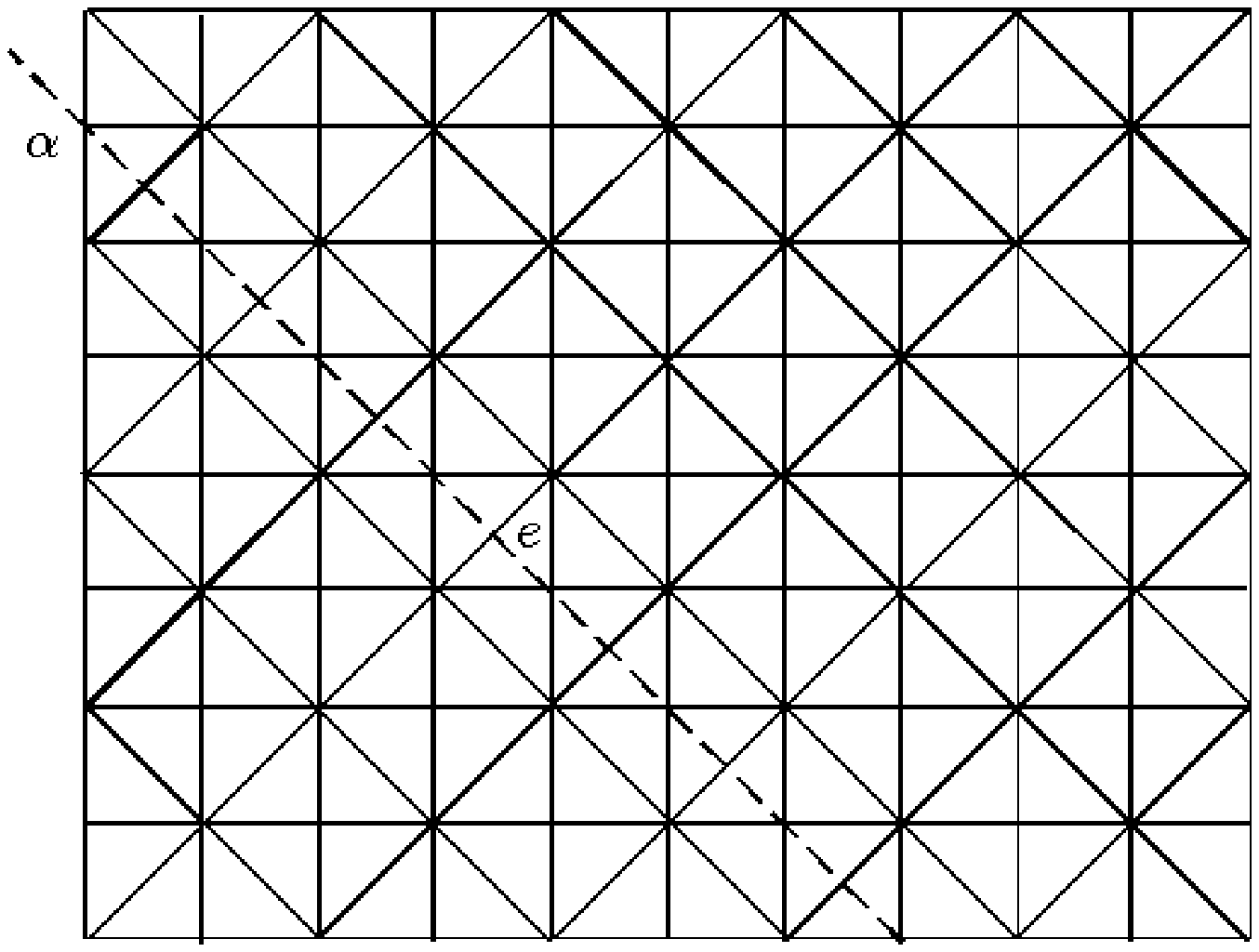}}

\centering{(a)}

	\begin{minipage}[c]{50mm}
		\centering{\includegraphics[width=50mm]{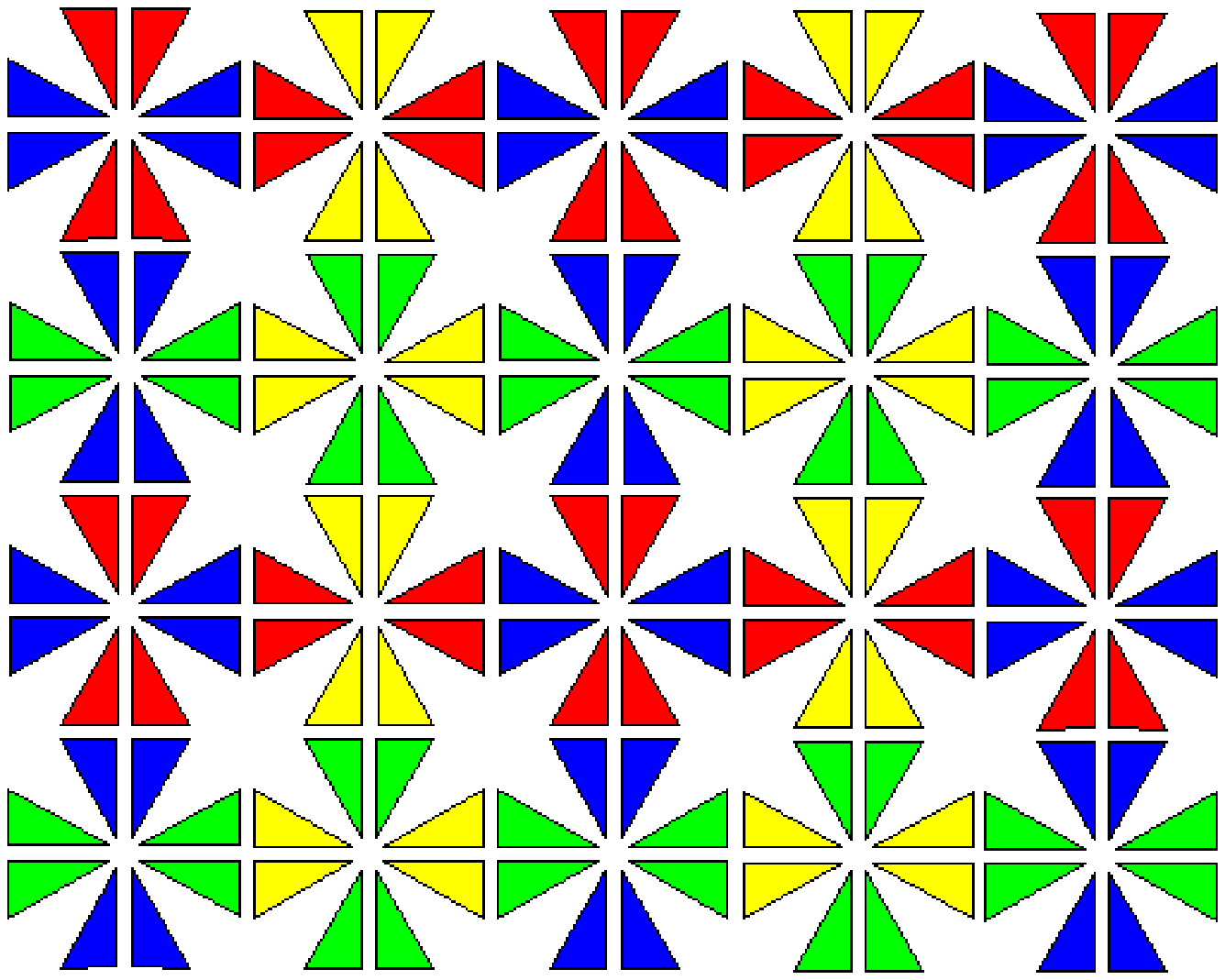}}
		
		\centering{(b)}
	\end{minipage}\quad
	\begin{minipage}[c]{50mm}
		\centering{\includegraphics[width=50mm]{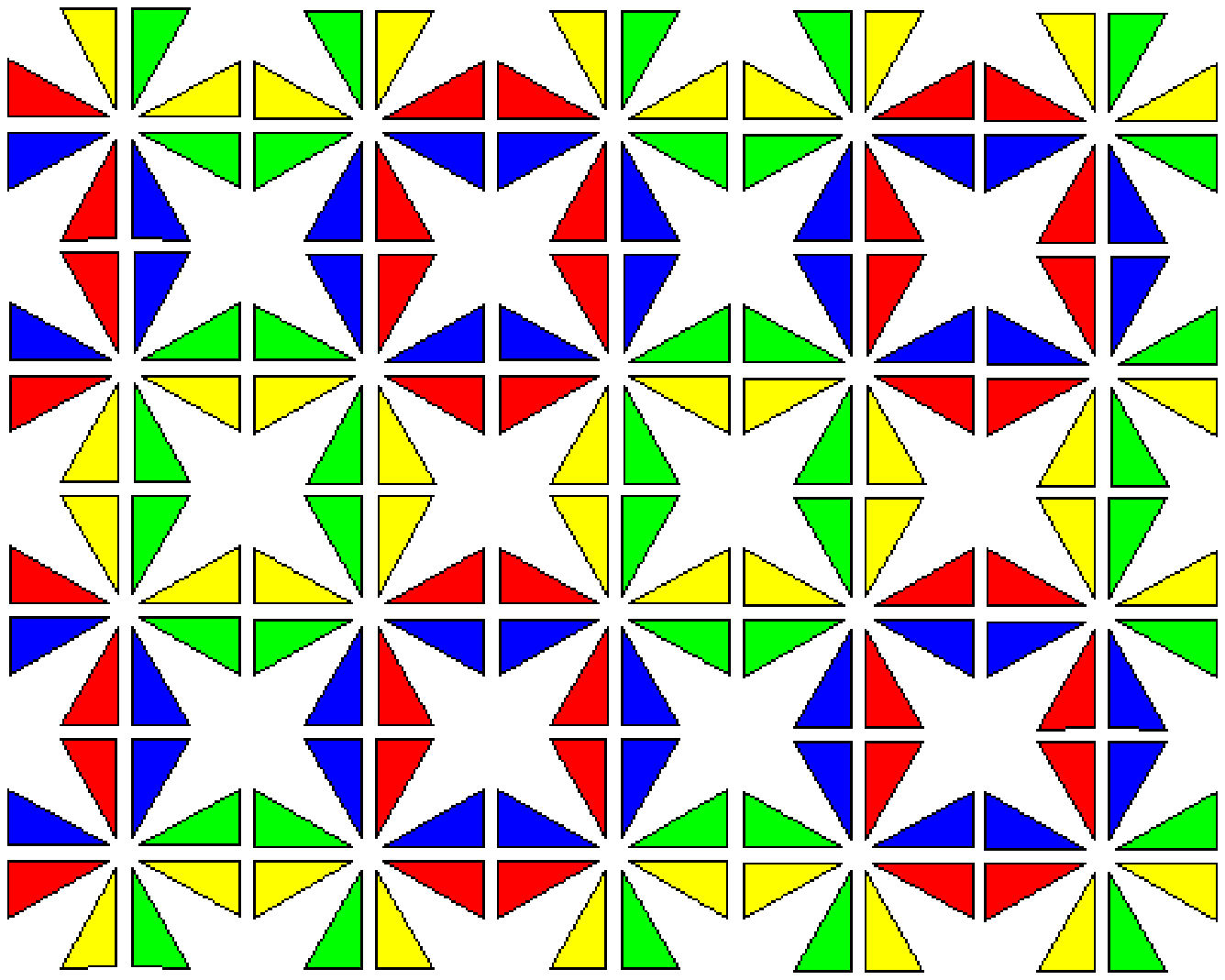}}

		\centering{(c)}
	\end{minipage}
\caption{(a) Fundamental domains of $G=\langle a,b,x,y\rangle$ of type $p4m$ and the reflection $\alpha\in N(G)$ ; (b) Semiperfect coloring
induced by the $(Y,J)-H$ partition of $G$, where $H=\langle a,ab,xy,x^{-1}y\rangle$ of type $p4m$, $J=\langle a^2b,b,x^2,y^2\rangle$ of type $pmm$, and $Y=\{e,xab\}$ ; (c) Semiperfect coloring induced by the $(Y^{\alpha},J^{\alpha})-H^{\alpha}$ partition of $G$}
\label{fig6}
\end{figurehere}

\medskip Therefore, to enumerate the semiperfect colorings of the hexagonal pattern in Figure \ref{fig1}(a), we may consider as $H$ only the subgroups $\langle a^2,b\rangle$ and $\langle a\rangle$ of $G$.  The resulting 25 inequivalent semiperfect colorings of the hexagonal pattern can be seen in \cite{L}.  In the case of infinite repeating patterns, we expect that there will be infinitely many semiperfect colorings.  However, one may impose certain restrictions on the colorings, such as the number of colors.  For instance, the 44 inequivalent semiperfect colorings of the infinite repeating pattern in Figure \ref{fig2} with at most four colors, one orbit of colors, and associated color group of type $p4m$, are listed in \cite{L}.

\section{Conclusion and Outlook}
In this paper, we considered semiperfect colorings of symmetrical patterns where the objects to be colored in the pattern are in a one-to-one correspondence with the elements of the symmetry group $G$ of the pattern.  In order to enumerate all inequivalent semiperfect colorings of the pattern, we looked at the different $(Y_i,J_i)-H$ partitions of $G$ where $H$ is a subgroup of index 2 in $G$.  We provided an organized and efficient method of identifying and counting the inequivalent $(Y_i,J_i)-H$ partitions of $G$ that correspond to semiperfect colorings.  Moreover, inequivalent semiperfect colorings whose associated color groups are conjugate subgroups with respect to $N(G)$ were related by considering the corresponding partitions and the images of the fundamental domains of $G$ under some suitable element of $N(G)$.

Unfortunately, not all colorings of symmetrical patterns correspond to a partition of the symmetry group of the pattern.  The next step would be to determine how to enumerate inequivalent semiperfect colorings of such patterns, examples of which are the Archimedean and hyperbolic tilings.  Adapting results in this paper to the framework discussed in \cite{DLPFL} might be effective in achieving this goal.

Semiperfect colorings is just a small part of the broader theory of nonperfect colorings.  One might look at the general case, that is, how to obtain all inequivalent colorings of a given symmetrical pattern whose associated color group is of index $n$ in the symmetry group of the pattern.

Lastly, it might be interesting to explore further colorings that are not really equivalent and yet one can be obtained from the other by a symmetry in the normalizer of the symmetry group of the pattern in the group of isometries.

\medskip\noindent\small\emph{Acknowledgments}.  The second author gratefully acknowledges the financial assistance given by the University of the Philippines HRDO. Part of this work was carried out by the second author during his stay at the FSPM, University of Bielefeld, Germany.

\end{document}